\documentclass[11pt]{article} 
\usepackage[margin=1in]{geometry}
\usepackage[utf8]{inputenc}
\usepackage[T1]{fontenc}
\usepackage[english]{babel}
\usepackage{lmodern}
\usepackage{graphicx}
\usepackage{wrapfig}
\usepackage{subfigure}
\usepackage{caption}
\usepackage{paralist}
\usepackage{enumitem}

\usepackage[linesnumbered,ruled,vlined]{algorithm2e}

\graphicspath{{./figures/}}

\makeatletter
\renewcommand{\p@enumii}{}
\makeatother

\newcommand{\titel}{Tight gaps in the cycle spectrum of 3-connected planar graphs}

\usepackage[usenames]{color} 
\definecolor{hellblau}{rgb}{0.2,0.4,1} 
\definecolor{dunkelblau}{rgb}{0,0,0.8}
\definecolor{dunkelgruen}{rgb}{0,0.5,0}
\usepackage[
	pdftex,
	colorlinks,
	linkcolor=dunkelblau,
	urlcolor=dunkelblau,
	citecolor=dunkelgruen,
	bookmarks=true,
	linktocpage=true,
	pdfsubject={}
]{hyperref} 
\usepackage{pdfpages}
\usepackage{amsmath}
\usepackage{amsthm}
\allowdisplaybreaks
\usepackage{amsfonts}
\theoremstyle{plain}
\newtheorem{satz}{Satz}[]
\newtheorem{theorem}[satz]{Theorem}
\newtheorem{lemma}[satz]{Lemma}

\newtheorem{proposition}[satz]{Proposition}

\theoremstyle{remark}

\theoremstyle{definition}

\newtheorem{conjecture}[satz]{Conjecture}

\begin{document}
	\title{\titel}
		\author{
			Qing Cui\thanks{Department of Mathematics, Nanjing University of Aeronautics and Astronautics, Nanjing, Jiangsu 210016, PR China. This work was partially supported by the Fundamental Research Funds for the Central Universities (No. NS2020055). E-mail address: \emph{cui@nuaa.edu.cn}} \and
			On-Hei Solomon Lo\thanks{School of Mathematical Sciences, University of Science and Technology of China, Hefei, Anhui 230026, PR China. This work was partially supported by NSFC grant 11622110. E-mail address: \emph{oslo@ustc.edu.cn}}\\
		}
	\date{}
	\maketitle

\begin{abstract}
	For any positive integer $k$, define $f(k)$ (respectively, $f_3(k)$) to be the minimal integer $\ge k$ such that every 3-connected planar graph $G$ (respectively, 3-connected cubic planar graph $G$) of circumference $\ge k$ has a cycle whose length is in the interval $[k, f(k)]$ (respectively, $[k, f_3(k)]$).
	
	Merker showed that $f_3(k) \le 2k + 9$ for any $k \ge 2$, and $f_3(k) \ge 2k + 2$ for any even $k \ge 4$. He conjectured that $f_3(k) \le 2k + 2$ for any $k \ge 2$. This conjecture was disproved by Zamfirescu, who gave an infinite family of counterexamples for every even $k \ge 6$ whose graphs have no cycle length in $[k, 2k + 2]$, i.e.\ $f_3(k) \ge 2k + 3$ for any even $k \ge 6$. However, the exact value of $f_3(k)$ was only known for $k \le 4$, and it was left open to determine $f_3(k)$ for $k \ge 5$. In this paper we improve Merker's upper bound, and give the exact value of $f_3(k)$ for every $k \ge 5$. We show that $f_3(5) = 10$, $f_3(7) = 15$, $f_3(9) = 20$, and $f_3(k) = 2k + 3$ for any $k = 6, 8$ or $\ge 10$.
	
	For general 3-connected planar graphs, Merker conjectured that there exists some positive integer $c$ such that $f(k) \le 2k + c$ for any positive integer $k$. We give a complete positive answer to this conjecture. We prove that $f(k) = 5$ for any $k \le 3$, $f(4) = 10$, and $f(k) = 2k + 3$ for any $k \ge 5$.
\end{abstract}

\begin{center}
\end{center}

\section{Introduction}

The \emph{cycle spectrum} of a graph $G$ is the set of lengths of cycles in $G$. For positive integers $3 \le a \le b$, the interval $[a, b]$ is a \emph{gap} in the cycle spectrum of $G$ if $G$ has circumference at least $a$ but has no cycle of length in $[a, b]$.
Recently, it was initiated by Merker~\cite{Merker2021} to study gaps in the cycle spectrum of 3-connected planar graphs. For any positive integer $k$, we define $f(k)$ to be the minimal integer $\ge k$ such that every 3-connected planar graph $G$ of circumference $\ge k$ has a cycle whose length is in the interval $[k, f(k)]$. So, $f(k)$ is the minimal integer $\ge k$ such that $[k, f(k)]$ is not a gap in the cycle spectrum of any 3-connected planar graph of circumference $\ge k$. Analogously, $f_3(k)$ is defined to be the minimal integer $\ge k$ such that $[k, f_3(k)]$ is not a gap in the cycle spectrum of any 3-connected cubic planar graph of circumference $\ge k$. By definition, $f(k) \ge f_3(k)$ for any positive integer $k$.

Merker~\cite{Merker2021} showed that $f_3(k) \le 2k + 9$ for any $k \ge 2$, hence it is always assured that there exists some cycle in $G$ of length in the interval $[k, 2k + 9]$. He also gave a construction which yields a lower bound $f_3(k) \ge 2k + 2$ for any even $k \ge 4$. And he proposed the following conjectures.

\begin{conjecture}[Merker~\cite{Merker2021}] \label{conj:1}
	$f_3(k) \le 2k + 2$ for any $k \ge 2$.
\end{conjecture}

\begin{conjecture}[Merker~\cite{Merker2021}] \label{conj:2}
	There exists some positive integer $c$ such that $f(k) \le 2k + c$ for any positive integer $k$.
\end{conjecture}

Very recently, Zamfirescu~\cite{Zamfirescu2020} gave an infinite family of counterexamples for every even $k \ge 6$ whose graphs have no cycle length in $[k, 2k + 2]$, i.e.\ $f(k) \ge f_3(k) \ge 2k + 3$ for any even $k \ge 6$. This improved the lower bound given by Merker, and disproved Conjecture~\ref{conj:1}. Note that Conjecture~\ref{conj:1} holds for $2 \le k \le 5$~\cite{Merker2021, Zamfirescu2020}.

However, the exact value of $f_3(k)$ was only known for $k \le 4$, namely, $f_3(k) = 5$ for $k \le 3$, and $f_3(4) = 10$. (The cases for $k \le 3$ are trivial; see Proposition~1 in~\cite{Merker2021} for the case $k = 4$.) It was left open to determine $f_3(k)$ for $k \ge 5$. In this paper we study the exact value of $f_3(k)$ for all the remaining cases. Indeed, more than that, we extend our attention to finding the exact value of $f(k)$. We will determine $f(k)$ for every positive integer $k$, and show that Conjecture~\ref{conj:2} is true. It is trivial to see that $f(k) = 5$ for $k \le 3$, also we have $f(4) \ge f_3(4) = 10$.

We give upper bound results in Section~\ref{sec:proof}. Confirming Conjecture~\ref{conj:2}, we show that $f_3(k) \le f(k) \le 2k + 3$ for $k \ge 3$ and $f(4) \le 10$ (Theorem~\ref{thm:main}), which is achieved by refining and generalizing Merker's proof. Combining the lower bounds mentioned in the previous two paragraphs, we immediately have $f(k) = f_3(k) = 2k + 3$ for every even $k \ge 6$, and $f(4) = 10$. 
For $k = 5, 7, 9$, we show $f_3(k) \le \frac{5(k - 1)}{2}$ (Theorem~\ref{thm:main579}). 
In Section~\ref{sec:odd11} we give a series of constructions which helps filling up all the unknown values of $f(k)$ and $f_3(k)$ so that we have the following complete list:
$$f(k) = 5 \textrm{ for } k \le 3, f(4) = 10 \textrm{ and } f(k) = 2k + 3 \textrm{ for } k \ge 5;$$ $$f_3(k) = 5 \textrm{ for } k \le 3, f_3(4) = 10, f_3(k) = \frac{5(k - 1)}{2} \textrm{ for } k = 5, 7, 9$$
$$\textrm{ and } f_3(k) = 2k + 3 \textrm{ for } k = 6, 8 \textrm{ or }\ge 10.$$ This gives us a characterization of the intervals which are gaps of 3-connected planar graphs, and that for 3-connected cubic planar graphs.

\begin{proposition}
	For any integers $3 \le a \le b$, the interval $[a, b]$ is a gap of some $3$-connected planar graph if and only if $a = 3$, $b \le 4$; or $a = 4$, $b \le 9$; or $a \ge 5$, $b \le 2a + 2$.
	
	For any integers $3 \le a \le b$, the interval $[a, b]$ is a gap of some $3$-connected cubic planar graph if and only if $a = 3$, $b \le 4$; or $a = 4$, $b \le 9$; or $a \in \{5, 7, 9\}$, $b \le \frac{5(a - 3)}{2}$; or $a \in \{6, 8\} \cup \{i \in \mathbb{N}: i \ge 10\}$, $b \le 2a + 2$.
\end{proposition}

\section{Upper bound results} \label{sec:proof}

We will make use of the tool developed by Merker~\cite{Merker2021}. We first prove a structural result, which generalizes Lemmas~5 and~6 in~\cite{Merker2021}. We remark that the condition that $G$ is a cubic graph is not necessary for the proofs of Lemmas~5 and~6 in~\cite{Merker2021}. Thus the method of Merker can be extended to general 3-connected planar graphs as follows. We give the proof below for completeness, although it is very similar to the proofs of Lemmas~5 and~6 in~\cite{Merker2021}.

For any fixed positive integer $k$, a cycle is \emph{short} if it has length less than $k$; otherwise it is \emph{long}. A face is \emph{short} (\emph{long}) if its boundary cycle is short (long). 

\begin{lemma} \label{lem:main}
	Let $k \ge 3$ be an integer. Every $3$-connected plane graph $G$ with circumference $\ge k$ but with no cycle of length in $[k, 2k]$ contains a $2$-connected subgraph $G'$ such that
	\begin{enumerate}[label=\emph{\textbf{(\Alph*)}}]
		\item no two short facial cycles of $G'$ share a common edge;
		\item $G'$ contains a long facial cycle;
		\item every long facial cycle in $G'$ is also a facial cycle in $G$ (of the same length);
		\item if the intersection of two long facial cycles of $G'$ is not empty, then it is either a vertex or an edge;
		\item and for every bounded (respectively, unbounded) short face $F$ in $G'$, every face in $G$ that is in the interior (respectively, exterior) of $F$ and contains an edge from the boundary of $F$ is also short.
	\end{enumerate}
\end{lemma}
\begin{proof}
	We first show that $G$ contains a long facial cycle. Suppose to the contrary that every facial cycle is short. Let $C$ be a long cycle with minimal number of faces in its interior. We call a cycle $U$ in the interior of $C$ a \emph{side cycle} if the intersection of $U$ and $C$ is precisely a non-trivial path. As $C$ is not a facial cycle, it is clear that there exists some side cycle. Let $U$ be a side cycle with minimal number of interior faces. And it is easy to see that $U$ is indeed a facial cycle. Hence we have that $U$ is a short cycle, and the symmetric difference of $C$ and $U$ is a cycle of length at least $(2k + 1) - (k - 1) > k$. This implies that we have a long cycle with fewer interior faces than $C$, contradicting the choice of $C$.
	Thus $G$ has a long facial cycle, which we denote by $C$.
	
	We now construct a 2-connected subgraph $G'$ of $G$ that contains the long facial cycle $C$ and satisfies the other required properties. This can be done by iteratively eliminating adjacent short faces. Initially, set $G_0 := G$. When $G_i$ is defined for some $i \ge 0$, we set $G' := G_i$ if there are no two short faces sharing an edge on their boundaries. Otherwise, $G_i$ has two short facial cycles $C_1, C_2$ that have some common edge. We delete all the common edges of $C_1$ and $C_2$, and also any resulting isolated vertices. This gives us a subgraph $H_i$ of $G_i$. It is readily to see that the symmetric difference of $C_1, C_2$ is a union of short cycles $D_1, \dots, D_s$ such that $D_2, \dots, D_s$ are in the interior of $D_1$. Note that $H_i$ is the partition of $s$ subgraphs, namely, the subgraph induced by the exterior of $D_1$ and the $s - 1$ subgraphs each induced by the interior of some $D_j$ ($2 \le j \le s$). Among these $s$ subgraphs of $H_i$, we take the one that contains the long facial cycle $C$ to be $G_{i + 1}$. 
	
	Obviously, this process will end and give us a subgraph $G'$ of $G$. And it is not hard to see that every intermediate graph $G_i$ is 2-connected and satisfies ~(B),~(C) and~(E). Thus it is left to show that $G'$ satisfies~(A) and~(D). As we output $G'$ only if it has no any two short facial cycles sharing an edge, it satisfies~(A). Suppose there are two long facial cycles of $G'$ whose intersection is not empty. By~(C), these are also two facial cycles in $G$. If their intersection is neither a vertex nor an edge, it will be either a path of length $\ge 2$ or a family of at least two paths (which may be trivial). In any case we will have a cut of size two in $G$, contradicting that $G$ is 3-connected. Thus $G'$ satisfies~(D).
\end{proof}

Here we give upper bounds on $f(k)$ for any $k \ge 4$. To this end, we refine and generalize the proof of Theorem~2 in~\cite{Merker2021} with our previous lemma.

A \emph{subdivided path} is a path that has two endvertices of degree three, and has its internal vertices of degree two.

\begin{theorem} \label{thm:main}
	Let $k \ge 3$ be an integer. Every $3$-connected planar graph $G$ with circumference $\ge k$ has a cycle of length in $[k, 2k + 3]$. For $k = 4$, we have $G$ has a cycle of length in $[4, 10]$.
\end{theorem}
\begin{proof}
	Suppose to the contrary that $G$ has no cycle of length in $[k, 2k + 3]$. Let $G'$ be the subgraph of $G$ given by Lemma~\ref{lem:main}. We further obtain a graph $G''$ from $G'$ by two kinds of vertex splitting defined as follows. For any vertex $v$ of degree greater than three, if $v$ is not in any short facial cycle, we split $v$ into two vertices $v_1, v_2$ (i.e.\ replace $v$ by two adjacent vertices $v_1, v_2$ such that every neighbor of $v$ is joined to either $v_1$ or $v_2$) such that both $v_1$ and $v_2$ have degree at least three and planarity is preserved in the obvious way. Otherwise, let $w_1, w_2$ be two neighbors of $v$ such that $w_1 v w_2$ is a path in some short facial cycle, we split $v$ into $v_1, v_2$ such that $v_1$ is of degree three, whose neighbors are exactly $v_2$, $w_1$ and $w_2$. 
	In any case the new vertices $v_1$ and $v_2$ are of degree at least three. We successively split vertices of degree greater than three in $G'$ until every vertex is of degree at most three. It is easy to see that this process will terminate and give a 2-connected subcubic plane graph $G''$ whose family of faces can be naturally identified with that of $G'$. For $G''$ is subcubic, the intersection of any two facial cycles in $G''$ does not have any isolated vertex.
	
	Notice that the vertex split operation does not alter the set of vertices of degree two.
	Also, by~(A), the boundary length of every short face is preserved after any vertex split, while some long facial cycles may be lengthened. This, in particular, implies that a face has a short (long) boundary in $G'$ if and only if its corresponding face has a short (long) boundary in $G''$. And we emphasize that $G''$ has no \emph{facial cycle} of length in $[k, 2k + 3]$. From these observations and Lemma~\ref{lem:main}, one may deduce the following properties of $G''$:
	\begin{enumerate}[label={\textbf{(\Alph*$'$)}}]
		\item no two short facial cycles of $G''$ intersect;
		\item $G''$ has some long facial cycle;
		\item every long facial cycle in $G''$ corresponds to a long facial cycle in $G$ (possibly of shorter length);
		\item no two long facial cycles of $G''$ can have two edges in common.
	\end{enumerate}
	
	Now we consider the graph $H$ which is obtained from $G''$ by suppressing the vertices of degree two. It can be shown in exactly the same way as in the proof of~\cite[Theorem~2]{Merker2021}\footnote{See page~72, lines 12--28 in~\cite{Merker2021}.} that $H$ is simple.
	The faces in $G''$ obviously correspond to the faces in $H$. For any face $F$ in $G''$, denote by $l(F)$ and $l_H(F)$ the length of $F$ in $G''$ and the length of the corresponding face in $H$, respectively. We define $X$ to be the set of short faces in $G''$, and $Y$ the set of long faces in $G''$. Since $G''$ has no facial cycle of length in $[k, 2k + 3]$, we have $l(F) \le k - 1$ for any $F \in X$, and $l(F) \ge 2k + 4$ for any $F \in Y$. Denote $n := |V(H)|$, $x := |X|$ and $y := |Y|$.

	For $H$ is a 2-connected cubic plane graph, we have, by Euler's formula, $$x + y = \frac{n}{2} + 2.$$ As $H$ is a simple plane graph and $G''$ satisfies~(A$'$), we have $$n \ge \sum_{F \in X} l_H(F) \ge 3x.$$
	
	Note that a subdivided path in $G''$ is corresponded to an edge in $H$. It follows from~(A$'$) and~(D$'$) that every subdivided path in $G''$ is incident with either one face from $X$ and one from $Y$, or two faces from $Y$. In the latter case the subdivided path is simply an edge that is preserved when we suppress the vertices of degree two in $G''$ to obtain $H$.
	For every face $F$ in $G''$, let $\mathcal{P}_F$ be the set of subdivided paths contained in the boundary cycle of $F$. We have $|\mathcal{P}_F| \ge 3$ for any face $F$ in $G''$ as $H$ is simple. Let $\mathcal{P}_{X}$ be the union of $\mathcal{P}_F$ over all faces $F$ from $X$. We have $$|\mathcal{P}_X| = \sum_{F \in X} |\mathcal{P}_F| \ge 3x$$
	and hence \begin{align*}
		\sum_{F \in Y} l(F) 
		&= \sum_{F \in Y} l_H(F) + \sum_{F \in X} l(F) - |\mathcal{P}_{X}|\\
		&\le \sum_{F \in X \cup Y} l_H(F) - \sum_{F \in X} l_H(F) + \sum_{F \in X} l(F) - 3x\\
		&\le 3n - 3x + \sum_{F \in X} (k - 1) - 3x\\
		&= 3n + (k - 7)x.
	\end{align*}

	On the other hand, as $l(F) \ge 2k + 4$ for any face $F \in Y$, we have 
	\begin{align*}
	\sum_{F \in Y} l(F) \ge (2k + 4)y = (2k + 4) \left( \frac{n}{2} + 2 - x \right) > (k + 2)n - (2k + 4)x. 
	\end{align*}
	Combining two bounds, we have $$(k + 2)n - (2k + 4)x < 3n + (k - 7)x$$ which is equivalent to $$(k - 1)n < (3k - 3)x.$$
	This contradicts $n \ge 3x$ and hence completes the proof of the first statement.
	
	For $k = 4$, if $G$ has no cycle of length in $[4, 10]$, no two short faces, in this case triangular faces, share any edge (otherwise they form a cycle of length four). Thus, following the construction given in the proof of Lemma~\ref{lem:main} we have $G' = G$. So, $G''$ is a 3-connected cubic plane graph that has no facial cycle of length in $[4, 10]$. Note that the triangular faces of $G''$ are mutually vertex-disjoint. Thus if we contract every triangle, for any two adjacent edges in some long facial cycle in $G''$ at most one of them is being contracted. Therefore after contracting the triangles in $G''$ we obtain a 3-connected plane graph each of whose faces is of length $\ge \lceil (10 + 1) / 2 \rceil = 6$, which is clearly not possible. This completes the proof of the second statement.
\end{proof}

Next we show that for 3-connected cubic planar graphs and small odd $k$, a smaller interval is sufficient to guarantee non-empty intersection with cycle spectrum.

\begin{theorem} \label{thm:main579}
	Let $k \in \{5, 7, 9\}$. Every $3$-connected cubic planar graph $G$ with circumference $\ge k$ has a cycle of length in $[k, \frac{5(k - 1)}{2}]$.
\end{theorem}
\begin{proof}	
	Suppose to the contrary that $G$ has no cycle of length in $[k, \frac{5(k - 1)}{2}]$.
	Let $G'$ be the subgraph of $G$ given by Lemma~\ref{lem:main}. By exactly the same argument as in the proof of~\cite[Theorem~2]{Merker2021}, one can show that every facial cycle in $G'$ contains at least three subdivided paths.
	For any face $F$ of $G'$, denote by $l(F)$ the length of $F$ in $G'$. 
	Define $X$ to be the set of short faces in $G'$, and $Y$ the set of long faces in $G'$.
	We have $l(F) \le k - 1$ for any $F \in X$, and $l(F) \ge \frac{5(k - 1)}{2} + 1$ for any $F \in Y$. As $G'$ satisfies~(B), we may assume that the unbounded face in $G'$ is long. 
	Since $G'$ satisfies~(C), every face from $Y$ is also a face in $G$, and hence every face in $G$ but not from $Y$ must be in the interior of some face from $X$.
	
	We claim that any subdivided path in the boundary of any face in $X$ has length $\le \frac{k - 3}{2}$. Otherwise, let $F$ be a (bounded) face in $X$, and let $P := v_1 w_1 \dots w_p v_2$ be a subdivided path in the boundary cycle $C_F$ of $F$, with $p \ge \frac{k - 3}{2}$. Let $R$ be the set of vertices of degree two (in $G'$) which is in $C_F$ but not in $P$. Denote $r := |R|$. We have $$p + r \le l(F) - 3 \le k - 4.$$ Note that if $r = 0$ and $p > 1$, then $\{w_1, w_p\}$ is a cut of size two in $G$, which contradicts that $G$ is 3-connected. Thus we may assume $r > 0$ for $k \in \{7, 9\}$.
	
	For any $1 \le i \le p - 1$, since $G'$ satisfies~(E), the face $F_i$ of $G$ incident with the edge $w_i w_{i + 1}$ in the interior of $C_F$ has length less than $k$. If $F_i$ intersects with $C_F$ other than $w_i, w_{i + 1}$, then, as $G$ is cubic and 3-connected, $F_i$ does not intersect any vertex in $P$ except $w_i, w_{i + 1}$, and indeed $F_i$ is incident with at least two vertices in $R$. Furthermore, if there are $p' \ge 1$ such faces, we have $r \ge p' + 1$.
	
	For $k = 5$, we must have $p = 1$ and $r = 0$, which is obviously impossible.
	
	For $k \in \{7, 9\}$, we consider two cases depending on the length of $F$. If $l(F) = k - 1$, then there is some face $F_i$ whose boundary has empty intersection with $R$; otherwise $$p + r \ge \frac{k - 3}{2} + \left(\frac{k - 3}{2} - 1 + 1\right) = k - 3,$$ which is not possible. Therefore, we have two short cycles of $G$, namely, $C_F$ and the boundary of $F_i$ for some $1 \le i \le p - 1$ such that their intersection is precisely the edge $w_i w_{i + 1}$. As $l(F) = k - 1$, we may obtain a cycle of length in $[k, 2k - 4] \subset [k, 5(k - 1)/2]$ by taking their symmetric difference, a contradiction. If $l(F) < k - 1$, it is obviously not possible for $k = 7$ as we know that $r > 0$. For $k = 9$, as $p + r \le l(F) - 3$, it must be $p = 3$, $r = 1$ and $l(F) = k - 2$. Moreover, each of the faces $F_1, F_2$ has empty intersection with $R$; otherwise $r \ge 2$. We may assume one of $F_1, F_2$, say $F_1$, has boundary of length at least four; otherwise $\{w_1, w_3\}$ is a cut of size two in $G$. Similary as in the first case, we may obtain a cycle of length in $[k, 2k - 5] \subset [k, 5(k - 1)/2]$ from $C_F$ and the boundary of $F_1$ by taking their symmetric difference, a contradiction. This justifies our claim.
	
	We now use the discharging technique in the original graph $G$ to derive a contradiction. As $G$ is a 3-connected cubic planar graph, it is easy to deduce from Euler's formula that $$\sum (l_G(F) - 6) = -12,$$ where the summation is over all faces $F$ of $G$, and $l_G(F)$ denotes the length of $F$ in $G$. For any face $F$ in $G$, we assign $l_G(F)$ as its initial charge. We have only one rule of discharging: for every edge that is incident with faces $F_1, F_2$ in $G$ such that $F_1 \in Y$ and $F_2 \notin Y$ (which is, by~(E), a short face), we let $F_1$ pass one charge to $F_2$. Denote by $l_G^*(F)$ the charge of $F$ after discharging. As $\sum l^*_G(F) = \sum l_G(F)$, to derive a contradiction it suffices to show $\sum (l^*_G(F) - 6) \ge 0$.
	
	For a (bounded) short face $F_0$ in $G'$, i.e.\ $F_0 \in X$, we show that $\sum (l^*_G(F) - 6) \ge 0$, where the summation is over all faces of $G$ that are in the interior of $F_0$. Let $p_1 \ge 3$ be the number of vertices that is on the boundary of $F_0$ joining to the exterior of $F_0$, and $p_2$ be the number of vertices that is either on the boundary of $F_0$ joining to the interior or in the interior of $F_0$. By Euler's formula, the number of faces of $G$ in the interior of $F_0$ is $2 - (p_1 + p_2) + \frac{2p_1 + 3p_2}{2} - 1 = \frac{p_2}{2} + 1$. By~(A) and~(E), every edge in the boundary of $F_0$ is incident with a long face from $Y$ (in the exterior of $F_0$) and a short face (in the interior of $F_0$) in $G$, so the long face passes one charge to the short face across this edge during discharging. This implies $\sum l^*_G(F) = \sum l_G(F) + l(F_0)$, which is precisely twice the number of edges that are in the interior (including the boundary) of $F_0$. Hence we have $\sum l^*_G(F) = 2p_1 + 3p_2$ and \begin{align*}
		\sum (l^*_G(F) - 6) = \sum l^*_G(F) - 6 (p_2/2 + 1) = (2p_1 + 3p_2) - 6 (p_2/2 + 1) \ge 0.
	\end{align*}
	
	It now suffices to show $l^*_G(F) \ge 6$ for any $F \in Y$. Note that, by~(C), $l_G(F) = l(F) \ge \frac{5(k - 1)}{2} + 1$ for any $F \in Y$. Since every subdivided path in the boundary of a face in $X$ has length at most $\frac{k - 3}{2}$, we conclude that $$l^*_G(F) \ge \left\lceil \frac{l_G(F)}{\frac{k - 3}{2} + 1} \right\rceil  \ge \left\lceil \frac{\frac{5(k - 1)}{2} + 1}{\frac{k - 3}{2} + 1} \right\rceil = 6.$$ This thus completes the proof.
\end{proof}

\section{Lower bound results} \label{sec:odd11}

As $f(k) \ge f_3(k)$ for any positive integer $k$, the tight lower bounds on $f(k)$ and $f_3(k)$ for every even $k \ge 6$ were indeed given by Zamfirescu's construction~\cite{Zamfirescu2020}. In this section we provide constructions which yield the desired lower bounds on $f(k)$ and $f_3(k)$ for every odd $k \ge 5$, completing the missing cases.

Similar to the constructions given by Merker~\cite{Merker2021} and Zamfirescu~\cite{Zamfirescu2020}, we first start with a cubic plane graph $T$ comprised of cycles of lengths $l$ and $2l$, where $l$ can be chosen for any integer $\ge k + 2$, and a perfect matching, and then replace each vertex by some fragment to obtain a graph $G$. Precisely, here we consider the graph $T$ formed by two cycles of length $l$ and an odd number of cycles of length $2l$ such that these cycles are placed in the plane without crossing in the way that the innermost and outermost cycles are of length $l$, and they are joined appropriately by a perfect matching; see e.g.\ Figures~\ref{fig:1} and~\ref{fig:19}, the horizontal and vertical line segments represent the cycles and the matching edges, respectively, while the open ends on the left and right are to be joined in the natural way.

We will use the fragments depicted in Figures~\ref{fig:234} and~\ref{fig:2341}. In constructing 3-connected cubic planar graphs, we only use the operation in Figure~\ref{fig:234}(a) for $k = 7$; the operations in Figures~\ref{fig:234}(a) and~\ref{fig:234}(b) for $k = 9$; and all three operations in Figure~\ref{fig:234} for $k \ge 11$.
In constructing 3-connected planar graphs, we use the operations in Figure~\ref{fig:2341}.
For our purpose each of the fragments in Figure~\ref{fig:234} has $k$ vertices and each in Figure~\ref{fig:2341} has $k - 1$ vertices. Note that every fragment has circumference $k - 1$.

We first describe the construction of 3-connected cubic planar graphs. For any odd $k \ge 11$, the graph $T$ contains a family of seven disjoint cycles, two of them are of length $l$ and five of them are of length $2l$; see Figure~\ref{fig:1}. The way to replace the vertices in $T$ should be clearly explained by Figure~\ref{fig:234}. For instance, the three vertices on the bottom in Figure~\ref{fig:1} are to be replaced by the fragment in Figure~\ref{fig:234}(a) (without any rotation or reflection); and those on the top are to be replaced by the same fragment but with a rotation by $180^\circ$. By construction, any cycle in $G$ of length less than $k$ is contained in some fragment, and any other cycle in $G$ has length at least $2k + 3$, which is the face length of $F_1, F_2$ and $F_3$. In other words, $G$ has no cycle length in $[k, 2k + 2]$ and hence $f_3(k) \ge 2k + 3$.
For $k = 5$ (respectively, $k = 7$), one can consider the graphs obtained from any 3-connected cubic plane graph of girth five by replacing vertices with disjoint triangles (respectively, copies of the fragment in Figure~\ref{fig:234}(a)), which yields $f_3(5) \ge 10$ (respectively, $f_3(7) \ge 15$).
For $k = 9$, the graphs obtained from the graph $T$ in Figure~\ref{fig:19} (which has two cycles of length $l$ and three of length $2l$) by the operations in (a) and (b) defined in Figures~\ref{fig:234} proves $f_3(9) \ge 20$.

Since $f(k) \ge f_3(k)$ for any positive integer $k$, it is left to find lower bounds for $f(5), f(7)$ and $f(9)$. However, the following construction may give tight lower bounds on $f(k)$ for every odd $k \ge 5$. By considering the 3-connected planar graphs obtained from the graph $T$ in Figure~\ref{fig:19} by the two operations in Figure~\ref{fig:2341}, it is readily to show that $f(k) \ge 2k + 3$ for any odd $k \ge 5$.

\begin{figure} 
	\centering
	\includegraphics[scale = 0.7]{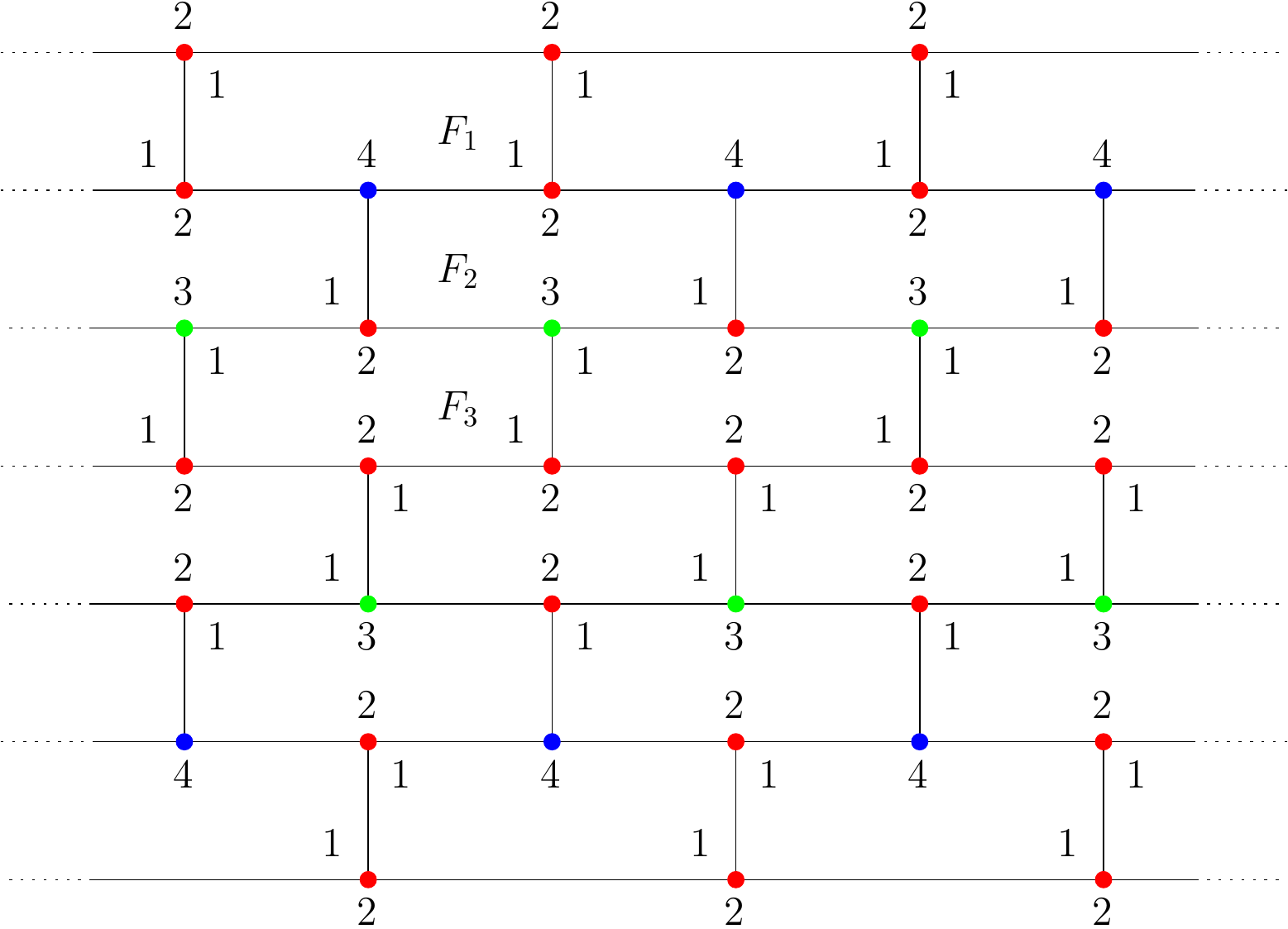}
	\caption{The base graph $T$ for constructing 3-connected cubic planar graphs in the case $k \ge 11$ is odd.}
	\label{fig:1}
\end{figure}

\begin{figure} 
	\centering
	\subfigure[]{\includegraphics[scale = 0.8]{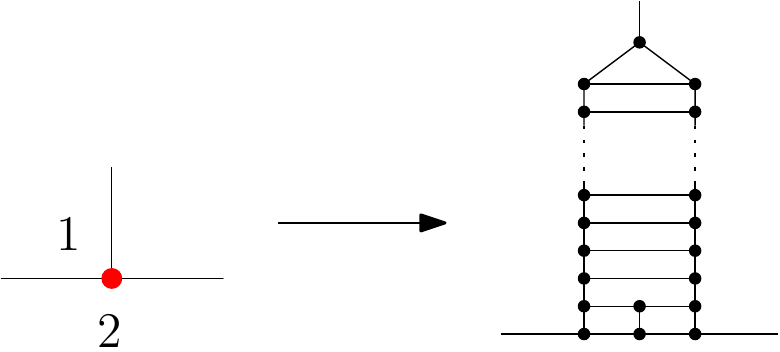}} \label{subfig:2}
	\hfil 
	\subfigure[]{\includegraphics[scale = 0.8]{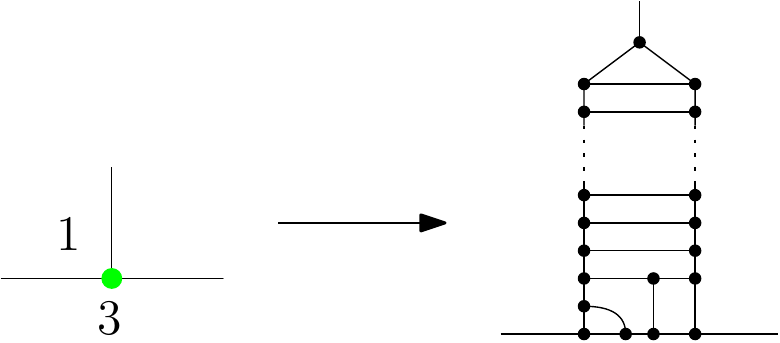}} \label{subfig:3}\\ 
	\subfigure[]{\includegraphics[scale = 0.8]{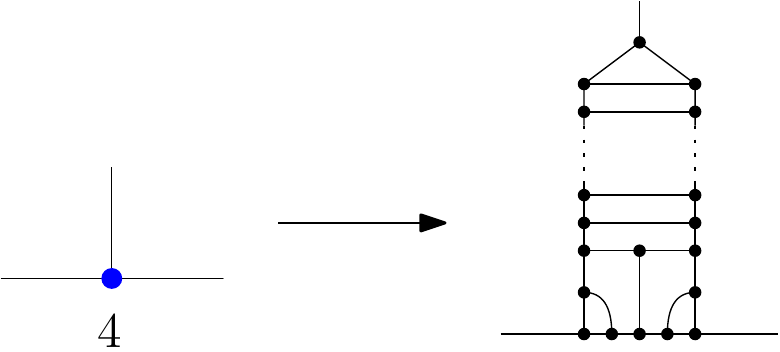}} \label{subfig:4}
	\caption{Three kinds of vertex replacements for constructing 3-connected cubic planar graphs in the case $k \ge 7$ is odd. The operations in~(b) and~(c) require $k \ge 9$ and $k \ge 11$, respectively. Every fragment has $k$ vertices.}
	\label{fig:234}
\end{figure}

\begin{figure} 
	\centering
	\includegraphics[scale = 0.7]{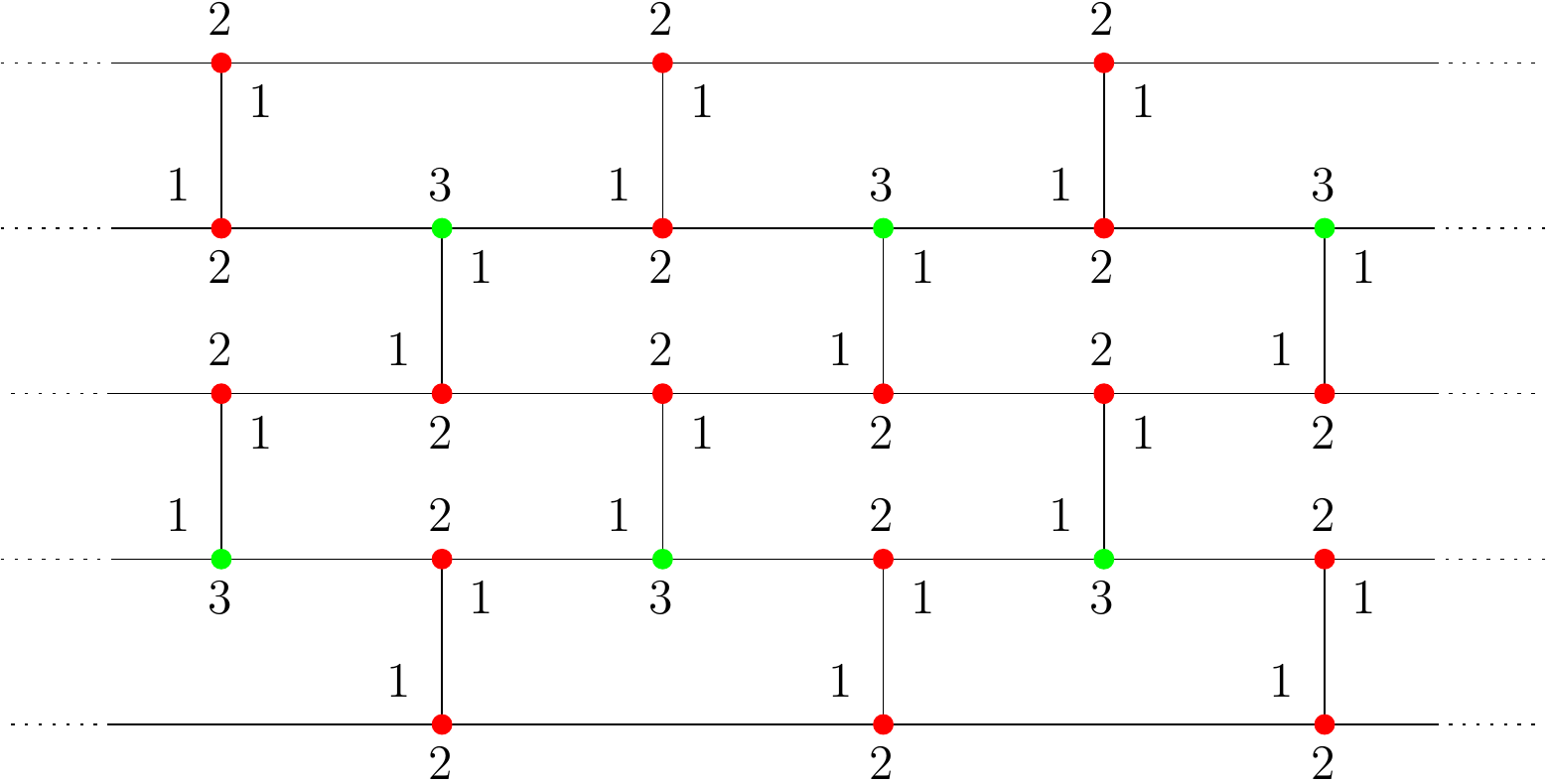}
	\caption{The base graph $T$ for constructing 3-connected cubic planar graphs in the case $k = 9$, and for 3-connected planar graphs in the case $k \ge 5$ is odd.}
	\label{fig:19}
\end{figure}

\begin{figure} 
	\centering
	\subfigure[]{\includegraphics[scale = 0.8]{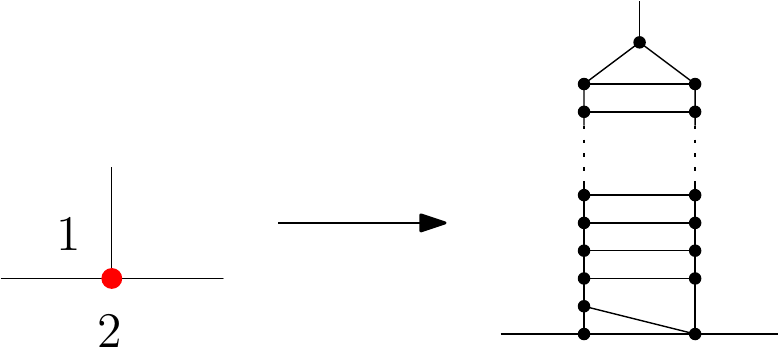}} \label{subfig:21}
	\hfil 
	\subfigure[]{\includegraphics[scale = 0.8]{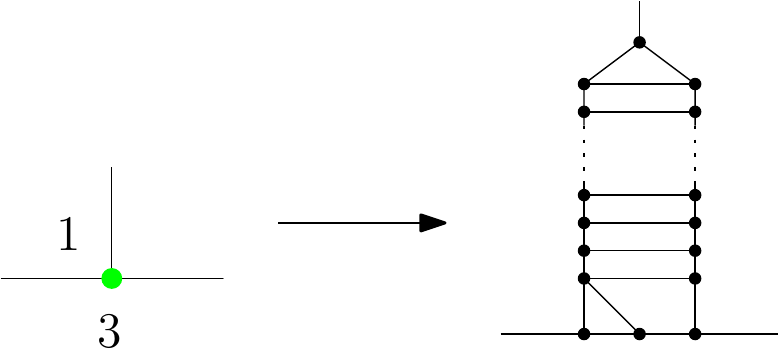}} \label{subfig:31}
	\caption{Two kinds of vertex replacements for constructing 3-connected planar graphs in the case $k \ge 5$ is odd. Every fragment has $k - 1$ vertices. Note that when $k = 5$, the fragment in~(b) is degenerated so that the vertex on top has degree four, and the three vertices on bottom have degree three.}
	\label{fig:2341}
\end{figure}


\bibliographystyle{abbrv}
\bibliography{Paper}

\begin{thebibliography}{1}

\bibitem{Merker2021}
M.~Merker.
\newblock Gaps in the cycle spectrum of 3-connected cubic planar graphs.
\newblock {\em J. Combin. Theory Ser. B}, 146:68--75, 2021.

\bibitem{Zamfirescu2020}
C.~T. Zamfirescu.
\newblock Counterexamples to a conjecture of {M}erker on 3-connected cubic
  planar graphs with a large cycle spectrum gap, 2020.
\newblock arXiv:2009.00423.

\end{thebibliography}

\end{document}